\documentclass[12pt, final]{amsart}
\usepackage{amsmath, amsxtra, amssymb}
\IfFileExists{mathdots.sty}{\usepackage{mathdots}}{}
\usepackage{verbatim}
\usepackage{multirow}

\usepackage{amsmath}
\usepackage{amssymb}
\usepackage{amscd}
\usepackage{latexsym}
\usepackage{amsthm}
\usepackage{mathrsfs}
\usepackage{color}
\usepackage{amsfonts}
\usepackage{enumitem}
\usepackage{array}
\usepackage{setspace}
\usepackage{mathtools}
\usepackage{tikz}
\usepackage{tikz-cd}
\usepackage[notcite, notref]{showkeys}
\IfFileExists{xurl.sty}{\usepackage{xurl}}{}
\usepackage[colorlinks, linkcolor=blue, citecolor=red, urlcolor=black]{hyperref}
\usetikzlibrary{arrows}
\usetikzlibrary{knots,patterns}

\usepackage[scale=0.75]{geometry}

\newtheorem{thm}{Theorem}[section]
\newtheorem*{thm*}{Theorem}
\newtheorem{cor}[thm]{Corollary}
\newtheorem{corollary}[thm]{Corollary}

\newtheorem*{cor*}{Corollary}
\newtheorem{lem}[thm]{Lemma}
\newtheorem*{lem*}{Lemma}
\newtheorem{lemma}[thm]{Lemma}
\newtheorem{claim}[thm]{Claim}

\newtheorem{prop}[thm]{Proposition}
\newtheorem*{prop*}{Proposition}

\theoremstyle{definition}
\newtheorem{defn}[thm]{Definition}
\newtheorem{definition}[thm]{Definition}

\newtheorem*{defn*}{Definition}
\newtheorem{conjecture}{Conjecture}
\newtheorem*{conjecture*}{Conjecture}

\newtheorem*{condition*}{Condition}
\newtheorem*{assumption*}{Assumption}

\theoremstyle{remark}
\newtheorem{rem}[thm]{Remark}
\newtheorem{remark}[thm]{Remark}

\newtheorem*{rem*}{Remark}

\newtheorem*{problem*}{Problem}

\numberwithin{equation}{section}

\newcommand\modp[2]{\langle #1\rangle_{#2}}

\def\co{\colon\thinspace} %macro to use in f\co \rightarrow Y
\def\QQ{\mathbb{Q}}
\def\RR{\mathbb{R}}

\def\ZZ{\mathbb{Z}}

\newcommand{\ra}{\rightarrow}

\newcommand{\BZ}{\mathbb Z}

\newcommand{\DD}{\mathbb D}

\newcommand{\M}{\overline M}

\newcommand{\wt}{\widetilde}

\DeclareMathOperator{\Sym}{Sym}

\DeclareMathOperator{\supp}{supp}

\DeclareMathOperator{\Pic}{Pic}
\DeclareMathOperator{\Span}{Span}
\DeclareMathOperator{\Fun}{Fun}

\DeclareMathOperator{\PBDcone}{PBDcone}

\begin{document}

\title{Symmetric and non-symmetric F-conjectures are equivalent}
\author{Maksym Fedorchuk}
\address[Fedorchuk]{Department of Mathematics, Boston College, 140 Commonwealth Ave, Chestnut Hill, MA 02467, USA}
\email{maksym.fedorchuk@bc.edu}
\author{Anton Mellit}
\address[Mellit]{Faculty of Mathematics, University of Vienna, 
Oskar-Morgenstern-Platz 1, 1090 Vienna, Austria}
\email{anton.mellit@univie.ac.at}
\date{\today}

\begin{abstract}
    The F-conjecture gives a conjectural description of the ample cone of the Deligne-Mumford moduli space $\M_{g,n}$. We prove that the $S_n$-symmetric and the non-symmetric F-conjectures are equivalent. We also prove the Strong F-conjecture for $\M_{0,8}$ (and give an alternative proof for $\M_{0,7}$).
    Finally, we derive, as a consequence, the F-conjecture for the moduli space of stable curves $\M_{g}$ up to genus $g\leq 44$.
\end{abstract}

\maketitle

\section{Introduction}

We describe some recent progress towards answering 
Mumford's question about ample line bundles on the Deligne-Mumford moduli space $\M_{g,n}$ of 
stable curves, conjecturally given by the well-known F-conjecture:

\begin{conjecture}[{F-conjecture for $\M_{g,n}$ \cite[Conjecture (0.2)]{GKM}}]\label{Cgn}
A divisor on $\M_{g,n}$ is nef if and only if it is F-nef. In particular, the nef cone is polyhedral. 
\end{conjecture}
Recall that a divisor on $\M_{g,n}$ is \emph{F-nef} if it intersects non-negatively with all the $1$-dimensional boundary strata of $\M_{g,n}$ (the F-curves).
The genus $0$ version is:
\begin{conjecture}[F-conjecture for $\M_{0,n}$]\label{C0n}
A divisor on $\M_{0,n}$ is nef if and only if it is F-nef.
\end{conjecture}

\begin{conjecture}[Symmetric F-conjecture]\label{Csym0n}
An $S_n$-symmetric divisor on $\M_{0,n}$ is nef if and only if it is F-nef.
\end{conjecture}

\begin{conjecture}[Strong F-conjecture {\cite[Question 0.13]{GKM}}]\label{Cstrong}
Every F-nef divisor on $\M_{0,n}$ is $\QQ$-linearly equivalent to an effective combination of boundary divisors.
\end{conjecture}

A breakthrough Bridge Theorem of Gibney-Keel-Morrison \cite[Theorem (0.3)]{GKM}
reduces Conjecture \ref{Cgn} for $\M_{g,n}$ to the $S_g$-symmetric version of Conjecture \ref{C0n} for $\M_{0,g+n}/S_g$, 
so that Conjecture \ref{Cgn} is true for all $g$ and $n$ if and only if Conjecture \ref{C0n} is true for all $n$ (if the characteristic is not $2$).

A standard restriction to the boundary argument \cite[Effective Dichotomy, p.39]{ian-mori} shows that Conjecture \ref{Cstrong} for all $n\leq N$ implies Conjecture \ref{C0n} for all $n\leq N$.
However, while Conjecture \ref{Cstrong} is known for $n=7$ by \cite{larsen} and $n\leq 6$ by \cite{FG}, 
Pixton found a counterexample for all $n\geq 12$ \cite{pixton}.

Because symmetric F-cones are much more manageable, and symmetric effective and F-nef cones are much simpler, one could hope that the symmetric F-conjecture in genus $0$ is weaker than the non-symmetric one. 
Our first result is that this is not so. The symmetric F-conjecture \ref{Csym0n} is in fact equivalent to the general F-conjecture \ref{Cgn} in the following sense:
\begin{thm}\label{T:equivalence}
Conjecture \ref{C0n} is true for all $n$ if and only if Conjecture \ref{Csym0n} is true for all $n$.
\end{thm}
By the Bridge Theorem, this immediately implies that Conjecture \ref{Cgn} reduces to Conjecture \ref{Csym0n}. That is, the general F-conjecture for $\M_{g,n}$
is implied by the symmetric F-conjecture in genus $0$:
\begin{cor} In any characteristic except $2$,  
Conjecture \ref{Cgn} is true for all $g$ and $n$ if and only if Conjecture \ref{Csym0n} is true for all $n$ if and only if Conjecture \ref{Cgn} is true for all $\M_g$.
\end{cor}
This resolves the problem posed in \cite[Remark, p.275]{GKM}.

Our second result is the proof of Conjecture \ref{Cstrong} (and hence \ref{C0n}) for $n=8$:
\begin{thm}[{Strong F-conjecture for $\M_{0,8}$}]
\label{n=8} 
Every F-nef divisor on $\M_{0,8}$ is $\QQ$-linearly equivalent to an effective combination of boundary divisors. In particular, the F-conjecture is true for $\M_{0,8}$.
\end{thm}

%Previously, Conjecture \ref{Cstrong} was known for $n=7$ by \cite{larsen} and $n\leq 6$ by \cite{FG}.
%Pixton found a counterexample for all $n\geq 12$ \cite{pixton}.

Our final result is:
\begin{thm}\label{T:strong-sym}
Suppose the Strong F-conjecture \ref{Cstrong} holds for $\M_{0,m}$ for all $m\leq k$. Then the symmetric F-conjecture \ref{Csym0n} is true for $\M_{0,n}$ for all $n\leq (k+1)(k+2)/2-1$. 
\end{thm}

In combination with Theorem \ref{n=8}, this implies:
\begin{cor} Suppose $g\leq 44$.
The symmetric F-conjecture holds for $\M_{0,\, g}$ in any characteristic. Hence the F-conjecture holds for $\M_g$ 
in any characteristic except $2$.
\end{cor}
The characteristic restriction in the second part is from \cite[Theorem (0.3)]{GKM}.

\subsection*{Acknowledgements} This work began (and Theorem \ref{T:equivalence} was proved) when the first author visited Universit\"at Wien in August 2023 with partial support from the Simons Foundation. The second author's research is supported by the Consolidator Grant No. 101001159 ``Refined invariants in combinatorics, low-dimensional topology and geometry of moduli spaces'' of the European Research Council. Boston College Andromeda HPC cluster
was used in verifying Theorem \ref{n=8}, which was first verified in April 2024.

\subsection*{AI disclosure}
All mathematical results in this paper, including the algorithm and original code verifying the Strong F-conjecture for $n=8$, were obtained by April 2024 without the use of AI tools. 
The original implementation ran for several days on the second author's personal computer, or roughly $8$ hours on $500$ CPU cores on Boston College's Andromeda HPC cluster. 
The refactored code used to produce the current certified verification (running in under $2$ hours on $250$ MPI ranks), with the accompanying Lean formalization, and
the entirety of Section \ref{S:n8-computation} were written with the assistance of GPT-6.2 Astra; we warmly thank Keerthi Madapusi for providing the ChatGPT workspace used for this.

\section{Generalities}

%\subsection{Notation and conventions:} 
We fix $n\geq 4$ and set $[n-1]:=\{1,2,\ldots,n-1\}$. Let $\Sigma_n$ be the set of non-empty subsets of $[n-1]$, 
or equivalently, the set of $2$-part partitions of $[n]:=\{1,2,\dots,n\}$;
specifically, given a $2$-part partition $I\sqcup J=[n]$, we take $I\in \Sigma_n$ to be the part such that $n\notin I$.

\subsection{Divisor classes}
\label{S:divisors}

%We note that $\Sigma_n$ parameterizes certain divisor classes on $\M_{0,n}$:
A $2$-part partition $I\sqcup J=[n]$ is \emph{proper} if $\vert I\vert, \vert J\vert \geq 2$. Proper partitions enumerate the boundary divisors, and 
the non-proper $2$-part partitions of $[n]$
correspond to the cotangent line bundles via a standard convention that $\Delta_{\{k\}, [n]\setminus \{k\}}:=-\psi_k$ for $k=1,\dots, n$:

\begin{enumerate}
%\item
%$I\sqcup J=[n]$ range over
%proper two part partitions, 
\item If $S\in \Sigma_n$ satisfies $2\leq |S| \leq n-2$, then $S$ corresponds to a proper partition $[n]=S \sqcup [n]\setminus S$, and 
the standard boundary class $\Delta_{S}:=\Delta_{S,  [n]\setminus S}$.
\item If $S=\{k\}\in \Sigma_n$ is a singleton, then $S$ corresponds to a non-proper partition $[n]=\{k\} \sqcup [n]\setminus \{k\}$, and we formally set $\Delta_{S}:=-\psi_k$, the negative of the $k^{th}$ cotangent line bundle.
\item If $S=[n-1]\in \Sigma_n$, then $S$ corresponds to a non-proper partition $[n]=[n-1] \sqcup \{n\}$, and we formally set $\Delta_{[n-1]}:=-\psi_n$, the negative of the $n^{th}$ cotangent line bundle.
\end{enumerate}
%\end{comment}

Let $\Fun(\Sigma_n)$ be the set of $\BZ$-valued functions on $\Sigma_n$. Then elements of $\Pic(\M_{0,n})$ are represented by elements of $\Fun(\Sigma_n)$ modulo relations, 
which we now define:
%\begin{comment}
    
For $i\in [n-1]$ and 
$S\in \Sigma_n$, let
\[
K_i(S) = 
\begin{cases}
	 1 & \text{if $i\in S$,}\\
	 0 & \text{otherwise.}
\end{cases} 
\]
We note that $K_i(S)$ is Keel's relation 
\begin{equation}\label{keel-1}
\psi_i+\psi_n=\sum_{i\in I, n\in J} \Delta_{I, J},
\end{equation}
where the sum on the right is over proper partitions.

For distinct $i,j\in [n-1]$ and $S \in \Sigma_n$, let
\[
K_{i,j}(S) =
\begin{cases}
	1 & \text{if $\{i,j\}\subset S$,}\\
	0 & \text{otherwise.}
\end{cases} 
\]
We note that $K_i(S)+K_j(S)-2K_{i,j}(S)$ is Keel's relation 
\begin{equation}\label{keel-2}
\psi_i+\psi_j=\sum_{i\in I, j\in J} \Delta_{I, J},
\end{equation}
where the sum on the right is over proper partitions.

We have $\Pic(\M_{0,n}) = \Fun(\Sigma_n) / \Span(K_i, K_{i,j}\;|\;i,j\in [n-1])$.

In terms of the generators $\{\Delta_{S}\mid  S\in \Sigma_n\}$ of $\Pic(\M_{0,n})$, a function $f\in \Fun(\Sigma_n)$ corresponds to the divisor class
\begin{equation}\label{Df}
D(f):=-\sum_{S\in \Sigma_n} f(S)\Delta_{S}.
\end{equation}
(Note the negative sign.)

\subsection{Attaching maps and stratally effective divisors}
Given a partition of $[n]$ into $k$ parts 
\[
[n]=A_1\sqcup \cdots \sqcup A_k,
\] 
there is a closed immersion $att_{(A_1,\dots,A_k)}\colon \M_{0,k}\to \M_{0,n}$ where $att_{(A_1,\dots,A_k)}(C,\{p_i\}_{i=1}^{k})$ is obtained from $C$ by attaching 
a fixed $(|A_i|+1)$-pointed (marked by $A_i$ and an extra attaching point $q_i$)
rational curve by identifying $p_i$ with $q_i$, and stabilizing. 

A divisor class $D\in \Pic(\M_{0,n})$ is called \emph{stratally effective boundary} if the pullback 
$att_{(A_1,\dots,A_k)}^* D$ is an \emph{effective boundary} (that is, represented by an effective $\QQ$-linear combination of boundary divisors) on $\M_{0,k}$ for all partitions 
of $[n]$, for all sizes $k$.

Note that if $D\in \Pic(\M_{0,n})^{S_n}$ is a symmetric divisor, then the class $att_{(A_1,\dots,A_k)}^* D$ depends only on the sizes $\lambda_i:=|A_i|$ of the parts, that is, only on the integer partition 
$\lambda:=  (\lambda_1,\dots, \lambda_k) \vdash n$. In this case, we write $att_{\lambda}^* D$ or $att^*_{(\lambda_1,\dots, \lambda_k)} D$ to denote this class.

The F-curves on $\M_{0,n}$ are precisely the images of $\M_{0,4}$ under the attaching maps corresponding to all partitions of $[n]$ into $4$ parts.

By a standard argument \cite[Effective Dichotomy, p.39]{ian-mori}, a stratally effective boundary divisor is nef, so that the Strong F-conjecture \ref{Cstrong} can be reformulated 
as saying that every F-nef divisor is a stratally effective boundary.

\subsection{Curve classes}\label{S:curves}

The curve classes on $\M_{0,n}$ are represented by their pairing with the divisor classes. They are linear combinations of F-curves: 
For sets $X, Y, Z\in \Sigma_n$ which pairwise do not intersect, we have an F-curve class $c_{X,Y,Z}$ corresponding to the partition of $[n]$ into $4$ parts
\[
[n]=X\sqcup Y\sqcup Z\sqcup \bigl([n] \setminus (X\cup Y \cup Z)\bigr).
\] 
For $f\in \Fun(\Sigma_n)$ defining the divisor $D(f)$, the intersection pairing $D(f)\cdot c_{X,Y,Z}$ is given by 
\[
\langle f, c_{X,Y,Z} \rangle := f(X) + f(Y) + f(Z) + f(X\cup Y\cup Z) - f(X\cup Y) - f(X\cup Z) - f(Y\cup Z).
\]
Notice that
\[
\langle K_i , c_{X,Y,Z} \rangle = \langle K_{i,j}, c_{X,Y,Z}  \rangle = 0.
\]

\subsection{F-nef functions}
A function $f\colon \Sigma_n \to \QQ$ is called F-nef if for all non-intersecting $X,Y,Z$ in $\Sigma_n$ we have
\[
\langle f, c_{X,Y,Z} \rangle \geq 0.
\]

Let $\Gamma$ be any abelian group. We say that a function $f\co \Gamma\to \QQ$ is symmetric if $f(x)=f(-x)$ for every $x\in \Gamma$. 
Then for every $n$-tuple $(d_1,\dots,d_n)$ of elements
of $\Gamma$ such that $d_1+\cdots+d_n=0 \in \Gamma$,
we define a divisor $\DD \bigl(\Gamma, f; (d_1,\dots,d_n)\bigr)$ on $\M_{0,n}$ by the formula 
\begin{equation}\label{E:divisor-intro}
\DD \bigl(\Gamma, f; (d_1,\dots,d_n)\bigr):=-\sum_{S\in \Sigma_n} f\Bigl(\sum_{i\in S} d_i\Bigr)\Delta_S=
\sum_{i=1}^n f(d_i) \psi_i -\sum_{I,J} f\Bigl(\sum_{i\in I} d_i \Bigr)\Delta_{I,J},
\end{equation}
where the last sum is over proper partitions.

For a symmetric function $f\colon \Gamma\to \QQ$ and $x,y,z\in \Gamma$, we formally define
\begin{equation}
\label{E:F-condition}
\langle f, (x,y,z)\rangle:=f(x)+f(y)+f(z)+f(x+y+z)- f(x+y) -  f(x+z) - f(y+z).
\end{equation}

\begin{definition}
A symmetric function $f\colon \Gamma\to \QQ$ 
 is called \emph{F-nef} %if $f(x)=f(-x)$ for all $x\in \Gamma$\footnote{We call such $f$ symmetric.} 
if for all elements $x,y,z\in \Gamma$ we have
\begin{equation}\label{eq:Fnef}
\langle f, (x,y,z)\rangle \geq 0,
\end{equation}
that is,
\begin{equation}\label{F-condition}
f(x)+f(y)+f(z)+f(x+y+z)\geq f(x+y) + f(x+z) + f(y+z).
\end{equation}
\end{definition}

F-nef functions on $\Sigma_n$ (respectively, on an abelian group $\Gamma$) form a convex cone in the space of all functions $\Sigma_n \to\QQ$ (respectively, the space of all functions $\Gamma\to \QQ$).

Importantly, given an F-nef function $f\colon \ZZ_m \to \QQ$ and $n$ divisible by $m$, the divisor
\begin{equation}\label{E:divisor-intro-symmetric}
\DD \bigl(\ZZ_m, f; (1,\dots,1)\bigr)=
f(1) \psi -\sum_{I,J} f(|J|)\Delta_{I,J}
\end{equation}
is an $S_{n}$-invariant F-nef divisor on $\M_{0,n}$.

\section{Standard, total, and supertotal F-nef functions}

In this section, we describe a family of F-nef functions on finite cyclic groups that we call \emph{standard functions}. These will be instrumental in our proof of Theorem \ref{T:equivalence}. 

Take $\Gamma=\ZZ_m$. For $k\in \ZZ$, we denote by $\modp{k}{m}$ the residue of $k$ in the set $\{0,1,\dots, m-1\}$. 
\begin{definition} \label{E:standard}
We define the \emph{standard function} $A_m\co \ZZ_m \ra \ZZ$ by
\begin{equation*}
A_{m}(i)=\modp{i}{m}\modp{m-i}{m}.
\end{equation*}
Given $a,b,c \in \ZZ_m$, set $d=\langle m-a-b-c\rangle_m$. It is easy to check that 
 \begin{align}\label{E:standard-function}
A_m(a)+A_m(b)&+A_m(c)+A_m(a+b+c)- A_m(a+b)-A_m(a+c)-A_m(b+c) \notag \\ &=\begin{cases} 0 \qquad \text{if $a+b+c+d=m$ \ \ or \ \ $a+b+c+d=3m$,}  \\
 2m \cdot \min\{ a, b, c, d, m-a, m-b, m-c, m-d \}, \ \text{otherwise.}
 \end{cases}
\end{align}
We conclude that $A_m$ is F-nef. 
\end{definition}

For any $k\in \BZ$, let
\[
A_{m,k}(x) = A_m(kx).
\]
This function is also F-nef. 

Set 
\begin{equation}\label{E:total}
T_m:=\sum_{k=1}^{m-1} A_{m,k},
\end{equation}
which we call \emph{the total standard function on $\ZZ_m$}.

We have
\begin{lem}\label{lem:improve once}
	Suppose $x,y,z\in\ZZ_m$ are such that $x,y,z,x+y+z\neq 0$.	Then 
\begin{equation}
\langle T_m, (x,y,z)\rangle=T_m(x) + T_m(y) + T_m(z) + T_m(x+y+z) - T_m(x+y)-T_m(x+z)-T_m(y+z)>0.
\end{equation}

\begin{comment}
    
there exists $k\in \BZ$ such that
	\[
	\varphi_k(x) + \varphi_k(y) + \varphi_k(z) + \varphi_k(x+y+z) - \varphi_k(x+y)-\varphi_k(x+z)-\varphi_k(y+z)>0.
	\]
    \end{comment}

\end{lem} 

\begin{comment}
\begin{rem} Let $w= \modp{-(x+y+z)}{m}$. If $x+y+z+w=2m$, then already $\langle A_{m,1}, (x,y,z)\rangle>0$. So we can assume without loss of generality that $x+y+z+w=m$.
Then the nef divisors obtained from the F-nef functions $A_{m,k}$, $k=1,\dots, m-1$ 
span the full-dimensional subcone of the symmetric effective cone of $\M_{0,m}$ by \cite{agss}.
The lemma then follows from this fact and the fact that 
there is an actual F-curve corresponding to the $4$-partition $x+y+z+w=m$ of $m$. 
\end{rem}
\end{comment}

\begin{proof}
	For any $k\in\BZ$, let $s_k=\modp{k x}{m}+\modp{k y}{m}+\modp{k z}{m}$. Replacing $x,y,z$ by $-x,-y,-z$ if necessary we may assume that $s_1<2m$. If $s_1>m$,  then already $\langle A_{m,1}, (x,y,z)\rangle>0$ by
    \eqref{E:standard-function}.
    %Lemma \ref{lem:phi}. 
    Suppose $s_1<m$. Then we have 
	\[
	s_{m-1}= 3m-s_1>2m.
	\]
	Let $k>0$ be the smallest integer satisfying $s_k>m$. So we have $s_{k-1}\leq m$. The inequality $\modp{kx}{m}\leq \modp{(k-1) x}{m} + \modp{x}{m}$, and similarly for $y, z$, implies
	\[
	s_k\leq s_{k-1}+s_1<2m.
	\]
	Let us show that $kx, ky, kz\neq 0$. Suppose the contrary. Without loss of generality assume $k x=0$. Then $\modp{(k-1) x}{m}=m-\modp{x}{m}$, and therefore
	\begin{multline}
	s_k = \modp{k y}{m} + \modp{k z}{m} \leq \modp{(k-1) y}{m} + \modp{(k-1) z}{m}  +\modp{y}{m} +\modp{z}{m} \\ \leq m-\modp{(k-1)x}{m}  +\modp{y}{m} +\modp{z}{m}  = s_1<m,
	\end{multline}
	a contradiction. So we have $kx,ky,kz\neq 0$ and $m<\modp{k x}{m}+\modp{k y}{m}+\modp{k z}{m}<2m$. Then $\langle A_{m,k}, (x,y,z)\rangle>0$ by
    \eqref{E:standard-function}.
    %Lemma \ref{lem:phi}.
\end{proof}

\subsection{Key Lemma}
Choose distinct prime numbers $p_1,p_2,\ldots,p_{n-1}$. Set 
\[
\Gamma = \prod_{i\in [n-1]} \ZZ_{p_i} \simeq \ZZ_{\prod_{i=1}^{n-1} p_i}.
\]
The components of $x\in\Gamma$ are denoted by $x_i\in\ZZ_{p_i}$. For each pair of distinct indices $i,j$, we denote by $x_{i,j}\in \ZZ_{p_ip_j}$ the element corresponding to $(x_i,x_j)$ via the isomorphism $\ZZ_{p_i p_j}\cong \ZZ_{p_i}\times \ZZ_{p_j}$.

Given $x\in \Gamma$, we define the support of $x$ to be 
\begin{equation}\label{E:support-x}
\supp(x):=\{i \mid x_i\neq 0 \in \ZZ_{p_i}\} \subset [n-1].
\end{equation}

We define \emph{the supertotal function} $ST$ on $\Gamma\simeq \ZZ_{p_1\cdots p_{n-1}}$ by (cf. \eqref{E:total})
\begin{equation}\label{E:supertotal}
ST(x)=\sum_{i=1}^{n-1} T_{p_i}(x_i)+\sum_{1\leq i <j \leq n-1} T_{p_ip_j}(x_{i,j}).
\end{equation}
%It is clearly an F-nef function on $\Gamma$. 

We have the following key lemma:

\begin{lemma}\label{L:supertotal-lemma}
Suppose $x,y,z,w$ 
are non-zero elements of $\Gamma$ satisfying $x+y+z+w=0\in \Gamma$. Then 
\begin{equation}\label{E:ST}
\langle ST, (x,y,z)\rangle = ST(x) + ST(y) + ST(z) + ST(w) - ST(x+y)-ST(x+z)-ST(y+z)\geq 0
\end{equation}
and the equality holds
if and only if $\supp(a)$, $\supp(b)$, and $\supp(c)$ are non-empty pairwise disjoint subsets in $[n-1]=\{1,\dots,n-1\}$ for some triple $\{a,b,c\} \subset \{x,y,z,w\}$.
\end{lemma}
\begin{proof}
The standard functions $T_*$ are all F-nef, so the inequality is clear. The content of the lemma is the condition on when the equality holds. 

First, suppose $\supp(a)$, $\supp(b)$, and $\supp(c)$ are non-empty pairwise disjoint subsets in $[n-1]=\{1,\dots,n-1\}$ for some triple $\{a,b,c\} \subset \{x,y,z,w\}$. Then for each pair of indices $i\neq j$, one of $\{a,b,c\}$, 
say $a$, satisfies $a_{i,j}=0$ so that 
by \eqref{E:standard-function}, we have
\[
\langle T_{p_i}, (a_i,b_i,c_i)\rangle = \langle T_{p_j}, (a_j,b_j,c_j)\rangle =\langle T_{p_ip_j}, (a_{i,j},b_{i,j},c_{i,j})\rangle = 0.
\] 
Thus $\langle ST, (a, b, c)\rangle =0$, as desired.

Conversely, suppose that the equality holds in \eqref{E:ST}.
    If for some $i\in [n-1]$ we have that $x_i, y_i, z_i, w_i$ are all non-zero, we get a contradiction by Lemma \ref{lem:improve once} 
    applied to the F-nef function $T_{p_i}$.

	Similarly, if for some $i\neq j$,
    all of $x_{i,j}, y_{i,j}, z_{i,j}, w_{i,j}$ are non-zero,
    we get a contradiction by Lemma \ref{lem:improve once} 
    applied to $T_{p_ip_j}$.
	
We are left with the situation when for each pair $i\neq j$ at least one out of $x_{i,j}, y_{i,j}, z_{i,j}, w_{i,j}$ is zero. 

For each $i\in [n-1]$,  let $Z_i\subset\{1,2,3,4\}$ be such that $1$, $2$, $3$, respectively $4\in Z_i$ if and only if $x_i=0$, $y_i=0$, $z_i= 0$, respectively $w_i= 0$. Then we have that $Z_i$ is non-empty for each $i$ and for each $i,j$ we have $Z_i\cap Z_j\neq\varnothing$.
	
	Observe that $\bigcap_{i\in [n-1]} Z_i = \varnothing$. Indeed, suppose $1\in Z_i$ for all $i\in [n-1]$. This means $x=0$, a contradiction.
	
	Notice that $|Z_i|$ cannot be $3$: if three out of $x_i, y_i, z_i, w_i$ are zero, the fourth one has to be zero too. Also $|Z_i|$ cannot be $1$: if $Z_i=\{r\}$ for some $i\in [n-1]$, we must have $r\in Z_j$ for all $j\in [n-1]$, a contradiction.
	
	So the only remaining possibility is that for each $i\in [n-1]$ we have $|Z_i|=2$ or $|Z_i|=4$.
    It is then easy to see that the union of all $Z_i$ of size $2$ has cardinality $3$.
    
    %Now we claim that there exists $r\in\{1,2,3,4\}$ such that $r\notin Z_i$ for all $i$ such that $|Z_i|=2$. Indeed, suppose say $Z_i=\{1,2\}$ and $Z_j=\{1,3\}$ (for this statement we can view $1,2,3,4$ up to a permutation). Suppose some $Z_\ell$ of size $2$ contains $4$. Since any two $Z_i$ must overlap, we have $Z_\ell=\{1,4\}$. Now any other $Z_q$ of size $2$ has to contain $1$. So all $Z_i$ contain $1$, a contradiction.
	
	So we have shown that there exists $r\in\{1,2,3,4\}$ such that for each $i\in [n-1]$ we have $|Z_i|=2$ and $r\notin Z_i$ 
    or $Z_i=\{1,2,3,4\}$. Permuting $x,y,z,w$, we can assume $r=4$. Denote 
	\[
	A=\{i\in [n-1]\;|\; Z_i=\{2,3\}\},\quad B=\{i\in [n-1]\;|\; Z_i=\{1,3\}\},\quad C=\{i\in [n-1]\;|\; Z_i=\{1,2\}\}.
	\] 
	These sets are non-empty and pairwise do not intersect. Clearly,
    \[\supp(x)=A, \quad \supp(y)=B,  \quad \supp(z)=C,\]
    which finishes the proof.
\end{proof}

\section{Proof of Theorem \ref{T:equivalence}}

\subsection{Main construction}
\label{S:main-construction}
Suppose $f$ is an F-nef function on $\Sigma_n$, the set of non-empty subsets of $[n-1]$. Choose distinct prime numbers $p_1,p_2,\ldots,p_{n-1}$. Take 
\[
\Gamma = \prod_{i\in [n-1]} \ZZ_{p_i} \simeq \ZZ_{p_1\cdots p_{n-1}}.
\]

%The components of $x\in\Gamma$ are denoted by $x_i\in\BZ/p_i\BZ$. For each pair of distinct indices $i,j$ we denote by $x_{i,j}\in \BZ/p_i p_j\BZ$ the element corresponding to $(x_i,x_j)$ via the isomorphism $\BZ/p_i p_j\BZ\cong \BZ/p_i\BZ \times \BZ/p_j\BZ$.

Define $\wt f\colon \Gamma\to \QQ$ by $\wt f(0)=0$ and 
\begin{equation}\label{E:f-tilda}
\wt f(x) = f(\supp(x)), \quad \text{if $x\neq 0$.}
\end{equation}
We clearly have $\wt f(x) = \wt f(-x)$ for all $x\in\Gamma$.

\begin{lemma}\label{L:main-lemma}
 $F:=\wt f +c \, ST$ is an F-nef function on $\Gamma$ for all large enough rational $c\gg 0$. 
 \end{lemma}

 \begin{proof} Consider a triple $x,y,z\in \Gamma$, and set $w=-x-y-z$.  If any of $x,y,z,w$ is zero, then $\langle F, (x,y,z)\rangle=0$ by evenness and $F(0)=0$.
 Hence we assume all of $x,y,z,w$ are non-zero.
 
 If $\langle ST, (x,y,z)\rangle>0$, then 
 \[
 \langle F, (x,y,z)\rangle= \langle \wt f, (x,y,z)\rangle+c  \langle ST, (x,y,z)\rangle >0.
 \]
 
 %Since $\Gamma$ is finite, one sufficiently large $c$ works for every triple
%$(x,y,z)$ with $\langle ST,(x,y,z)\rangle>0$.

 Suppose $\langle ST, (x,y,z)\rangle=0$. Then by Lemma \ref{L:supertotal-lemma}, we can 
 assume without loss of generality that $\supp(x), \supp(y), \supp(z)$ are \emph{non-empty pairwise disjoint subsets} of $[n-1]$,
 so that we have a well-defined 
 F-curve 
 $c_{\supp(x),\supp(y),\supp(z)}$ in $\M_{0,n}$. 
 
 We
 compute  
 \begin{multline}
 \langle F, (x,y,z)\rangle= \langle \wt f, (x,y,z)\rangle \\= f(\supp(x))+ f(\supp(y))+f(\supp(z))+f(\supp(x+y+z))\\ -f(\supp(x+y))-f(\supp(x+z))-f(\supp(y+z)) \\
 =f(\supp(x))+ f(\supp(y))+f(\supp(z))+f(\supp(x)\cup \supp(y) \cup \supp(z)) \\ -f(\supp(x)\cup \supp(y))-f(\supp(x) \cup \supp(z))-f(\supp(y) \cup\supp(z))\\
= \langle f, c_{\supp(x),\supp(y),\supp(z)}\rangle \geq 0,
  \end{multline}
  by the F-nefness of $f$.

 \end{proof}

\subsection{Every F-nef divisor is a pullback of a symmetric F-nef divisor}

Theorem \ref{T:equivalence} clearly follows from the following:
\begin{thm}\label{thm:main theorem}
	For every F-nef divisor class $D$ on $\M_{0,n}$, there exists an attaching map 
	\[
	att_{(a_1,\dots,a_n)}\colon \M_{0,n} \to \M_{0,\sum_i a_i}
	\] such that 
	$D$ is a pullback of an F-nef symmetric divisor on $\M_{0,\sum_i a_i}$ via $att_{(a_1,\dots,a_n)}$.
\end{thm}
 \begin{remark}
 We note the analogy of this result with the proof of \cite[Theorem (4.7)]{GKM}. 
 \end{remark}
 
\begin{proof}[Proof of Theorem~\ref{thm:main theorem}]

Choose distinct primes $p_1, \dots, p_{n-1}$ and set $\mathbf N:=\prod_{i=1}^{n-1} p_i$. Choose positive integers $\{a_i\}_{i\in [n-1]}$ such that 
\begin{equation}\label{E:CRT}
\begin{aligned}
a_i &\equiv 1 \pmod{p_i} \\ 
a_i &\equiv 0 \pmod{p_j} \ \text{for all $j\neq i$},
\end{aligned}
\end{equation}
and a positive integer $a_n$ such that $a_n \equiv -1 \pmod{p_i}$ for all $i\in [n-1]$. 

Set $\mathbf A:=\sum_{i=1}^n a_i$; note that $\mathbf N=p_1\cdots p_{n-1} \mid \mathbf A$. Consider the attaching map 
\[
att:=att_{(a_1,\dots,a_n)}\colon \M_{0,n} \to \M_{0,\mathbf A}
\]
%that attaches a fixed $(a_i+1)$-pointed fixed rational curves at the $p_i$.
where $att(C,\{p_i\}_{i=1}^{n})$ is obtained from $C$ by attaching a fixed $(a_i+1)$-pointed curve at every $p_i$.

Suppose $D$ is an F-nef divisor on $\M_{0,n}$.  Let $f\colon \Sigma_n\to \ZZ$ be an F-nef function on $\Sigma_n$ such that $D=D(f)$ (cf. \S\ref{S:divisors}).

Let $ST\colon  \ZZ_{\mathbf N} \to \ZZ$ be the supertotal function defined in \eqref{E:supertotal}, and $\wt{f}$ be as defined in \eqref{E:f-tilda}.
  Invoking Lemma \ref{L:main-lemma}, we obtain an F-nef function 
  \[
  F:=\wt f +c \, ST\colon \ZZ_{\mathbf N} \to \ZZ
  \]
  for some $c\gg 0$. 
  
  Since $\mathbf N \mid \mathbf A$, the F-nef function $F$ defines an F-nef $S_{\mathbf A}$-symmetric divisor 
  \begin{equation}
  \DD \bigl(\ZZ_{\mathbf N}, F; (1,\dots,1)\bigr)=\DD \bigl(\ZZ_{\mathbf N}, \wt f; (1,\dots,1)\bigr)+c \, \DD \bigl(\ZZ_{\mathbf N}, ST; (1,\dots,1)\bigr)
  \end{equation}
  on $\M_{0,\mathbf A}$.
  
\begin{claim}
We have 
\begin{align*}
att^*_{(a_1,\dots,a_n)} \DD \bigl(\ZZ_{\mathbf N}, \wt f; (1,\dots,1)\bigr)&=D(f)=D, \\
att^*_{(a_1,\dots,a_n)} \DD \bigl(\ZZ_{\mathbf N}, ST; (1,\dots,1)\bigr)&=0. 
\end{align*}
Consequently,
\begin{equation}
D=att^*_{(a_1,\dots,a_n)}  \DD \bigl(\ZZ_{\mathbf N}, F; (1,\dots,1)\bigr).
\end{equation}
\end{claim}

\begin{proof}[Proof of Claim]
Take an F-curve $c_{I,J,K}$ in $\M_{0,n}$, where $I, J$ and $K$ are non-empty pairwise disjoint subsets of $[n-1]$. Set $A(I):= \sum_{i\in I} a_i$, $A(J):= \sum_{j\in J} a_j$, and $A(K):= \sum_{k\in K} a_k$.
Note that by the choice of $a_i$'s in \eqref{E:CRT}, we have
\[
\supp(A(I))=I, \quad \supp(A(J))=J, \quad \text{and} \quad \supp(A(K))=K.
\]
Note also that for each $i\neq j \in [n-1]$, one of $A(I), A(J)$, or $A(K)$ is divisible by $p_ip_j$ (namely, take a subset not containing $i$ and $j$). It follows that 
\[
\langle T_{p_i}, \bigl(A(I), A(J), A(K)\bigr)\rangle = \langle T_{p_ip_j}, \bigl(A(I), A(J), A(K)\bigr)\rangle=0, \quad \text{for all $i\neq j\in [n-1]$}.
\]
Then 
\begin{align*}
\langle att^*_{(a_1,\dots,a_n)} \DD \bigl(\ZZ_{\mathbf N}, \wt f; (1,\dots,1)\bigr),  c_{I,J,K} \rangle&= \langle \wt f, \bigl(A(I), A(J), A(K)\bigr)\rangle=\langle f, c_{I,J,K}\rangle =D\cdot c_{I,J,K} \, , \\
\langle  att^*_{(a_1,\dots,a_n)} \DD \bigl(\ZZ_{\mathbf N}, ST; (1,\dots,1)\bigr), c_{I,J,K} \rangle&=\langle ST, \bigl(A(I), A(J), A(K)\bigr)\rangle=0.
\end{align*}
Since the classes of F-curves generate $N_1(\M_{0,n})$, this finishes the proof of the claim.
\end{proof}

This finishes the proof that $D$ is a pullback of an F-nef symmetric divisor on $\M_{0,\mathbf A}$. We also remark that the divisor class $\mathcal{D}:= \DD \bigl(\ZZ_{\mathbf N}, ST; (1,\dots,1)\bigr)$ is a symmetric \emph{semiample} divisor class on $\M_{0, \mathbf A}$ which satisfies 
$att^*\mathcal{D}=0$. Semiampleness follows by applying \cite[Theorem~5.3.9 and Example~5.3.10]{semiample-ant} to the standard-function summands of $ST$.
\end{proof}

\begin{remark}
Theorem \ref{thm:main theorem} shows that a symmetric F-nef divisor is not necessarily stratally effective boundary, thus disproving \cite[Conjecture 4.4]{moon-swinarski} and crushing an earlier hope of the first author.
Namely, feed Pixton's non-stratally effective boundary nef divisor on $\M_{0,12}$ into Theorem \ref{thm:main theorem} to obtain a non-stratally effective boundary symmetric F-nef divisor on some $\M_{0,\mathbf A}$ 
for some large value of $\mathbf A$.
\end{remark}

\section{Proof of Theorem \ref{T:strong-sym}}

Our key result is that in verifying stratal effectivity of symmetric divisors, it suffices to consider only partitions with distinct parts:
\begin{prop}\label{P:main}
An F-nef $S_n$-invariant divisor $D$ on $\M_{0,n}$ is stratally effective boundary if and only if $att_{(a_1,\dots, a_k)}^*D$ is an effective boundary for all \emph{strict} partitions
$n=a_1+\cdots+a_k$, with $a_1>a_2>\cdots >a_k \geq 1$.
\end{prop}

Since $n$ has a strict partition of size $> k$ if and only if $n\geq (k+1)(k+2)/2$, we obtain:
\begin{corollary}\label{Cor}
Suppose the Strong F-conjecture holds for $\M_{0,m}$ for all $m\leq k$. Then the symmetric F-conjecture holds for $\M_{0,n}$ for all $n\leq (k+1)(k+2)/2-1$. 
\end{corollary}

\subsection{Proof of Proposition \ref{P:main}}

The result follows from:
\begin{lemma}[Ascent of effectivity]\label{L:lemma} Let $D$ be an F-nef $S_n$-invariant divisor class on $\M_{0,n}$.
Suppose $n=\lambda_1+\cdots+\lambda_{k-1}+\lambda_k$ is a \emph{non-strict} partition %of $n$ 
of size $k$, say, with $\lambda_{k-1}=\lambda_k$.  
Set $\mu:=(\lambda_1,\dots,\lambda_{k-2}, 2\lambda_{k-1})\vdash n$.
If $att_{\mu}^* D$ is an effective boundary on $\M_{0,k-1}$, then $att_{\lambda}^*D$ is an effective boundary on $\M_{0,k}$.
\end{lemma}

Recall from \S\ref{S:divisors} that every divisor class $D$ on $\M_{0,m}$ can be written as
\begin{equation}\label{E:divisors}
D=-\sum_{I \sqcup J=[m]} b_{I,J} \Delta_{I,J}\ ,  %=\sum_{i=1}^{n} b_{\{i\}\sqcup [m]\setmimus \{i\}} \psi
\end{equation}
where the sum is taken over all $2$-part partitions of $[m]:=\{1,\dots,m\}$. 

The ambiguity in writing $D$ as in \eqref{E:divisors} is completely described 
by Keel's relations \eqref{keel-1}--\eqref{keel-2} in $\Pic(\M_{0,m})$ (see \cite[Theorem 2.2(d)]{AC} and \cite{keel}). We use the following formulation:
\begin{lemma}[Effective Boundary Lemma {\cite[Lemma 2.3.3]{semiample-ant}}]\label{L:boundary} 
We have
\[
-\sum_{I\sqcup J=[m]} b_{I,J} \Delta_{I,J} = \sum_{I\sqcup J=[m]} c_{I,J}\Delta_{I,J} \in \Pic(\M_{0,m})\otimes \QQ
\] 
if and only if there is a function $w\colon \Sym^2 \{1,\dots,m\}   \to \QQ$ 
such that for every $2$-part partition $I\sqcup J=[m]$, we have:
\begin{equation*}%\label{E:tilde-lambda}
\sum_{i\in I, j\in J} w(i,j) =c_{I,J}+b_{I,J}.
\end{equation*}
In particular, a divisor $D=-\sum_{I\sqcup J=[m]} b_{I,J} \Delta_{I,J}$
 is an effective boundary 
on $\M_{0,m}$ if and only if there exists a function $w$ such that 
\begin{align*}
\sum_{i\in I, j\in J} w(i,j)& \geq b_{I,J},
\end{align*}
for all partitions $I\sqcup J=[m]$, with equality holding for all non-proper partitions.
\end{lemma}

\begin{proof}[Proof of Lemma \ref{L:lemma}]
Let $f\colon \ZZ_n \to \QQ$ be an F-nef symmetric function such that 
\begin{equation}\label{E:L}
D=D(f)=-\sum_{I \sqcup J=[n]} f(\vert I\vert) \Delta_{I,J},
\end{equation}
where the sum is taken over all $2$-part partitions of $[n]$. %:=\{1,\dots,n\}$. 
%We will say that such 
%a partition is proper if $\vert I\vert, \vert J\vert \geq 2$. Proper partitions enumerate the boundary divisors, and 
%the non-proper two-part partitions of $[n]$
%correspond to the cotangent line bundles via a standard convention that $\Delta_{\{k\}\sqcup [n]\setminus \{k\}}:=-\psi_k$.
If $D=\sum_{i=2}^{\lfloor n/2\rfloor} c_i \Delta_i$ in the standard basis of $\Pic(\M_{0,n})^{S_n}$, then we can take
$\tilde{f}(i)=\tilde{f}(n-i)=-c_i$ for all $i=2,\dots, \lfloor n/2\rfloor$, and $\tilde{f}(0)=\tilde{f}(1)=\tilde{f}(n-1)=0$, 
and then $f=\tilde{f}+c A_n$, where $A_n$ is the standard function on $\ZZ_n$ defined in Definition~\ref{E:standard} and $c\gg 0$, as in \cite[Lemma~3.2.3]{semiample-ant}.

The function $f$ is a symmetric F-nef function on $\ZZ_n$ that satisfies the 
F-inequality: %, which we call the F-condition:
\begin{equation}\label{E:F-nef}
\langle f, (a,b,c)\rangle=f(a)+f(b)+f(c)+f(d)- f(a+b)-f(b+c)-f(a+c) \geq 0,
\end{equation}
for all $a,b,c\in \ZZ_n$, where $d=-a-b-c$ in $\ZZ_n$.

By the assumption, 
\[
att_{\mu}^*D %\sim c_{I,J} \Delta_{I,J}=
=
- \sum_{I\sqcup J=[k-1]} f\bigl(\sum_{t\in I} \mu_t\bigr) \Delta_{I,J} 
\] is an effective boundary on $\M_{0,k-1}$. 
By Lemma \ref{L:boundary}, there is a function $\widetilde{w}\colon \Sym^2 [k-1] \to \QQ$ 
such that for every $2$-part partition $I\sqcup J=[k-1]$, we have:
\begin{equation}\label{E:tilde-lambda}
\sum_{i\in I, j\in J} \widetilde{w}(i,j) \geq f(\sum_{t\in I} \mu_t),
\end{equation}
%for all  partitions $I\sqcup J=[k-1]$ 
%and 
with equality holding for 
all non-proper partitions. % of $[k-1]$. 

Define $w\colon \Sym^2 [k]  \to \QQ$ by 
\begin{align*}
w(i,j)&=\begin{cases}
\widetilde{w}(i,j) & \text{if $i,j\in \{1,\dots, k-2\}$} \\
\widetilde{w}(k-1,j)/2 & \text{if $i\in \{k-1,k\}$ and $j\notin \{k-1,k\}$} 
\end{cases} \\
w(k-1,k)&=f(\lambda_k)-\frac{1}{2}f(2\lambda_{k})=f(\lambda_k)-\frac{1}{2}f(\mu_{k-1}). 
\end{align*}
Since
%\begin{equation*}
$
att_{\lambda}^*D = - \sum_{I\sqcup J=[k]} f\bigl(\sum_{t\in I} \lambda_t\bigr) \Delta_{I,J} \ , 
$
%\end{equation*}
Lemma \ref{L:boundary} implies that $att_{\lambda}^*D$ is an effective boundary on $\M_{0,k}$ once we establish the following:
\begin{claim}\label{C:claim} For all $2$-part partitions $I\sqcup J=[k]$, we have
\begin{equation}\label{E:lambda}
\sum_{i\in I, j\in J} w(i,j) \geq f(\sum_{t\in I} \lambda_t), 
\end{equation}
with equality holding for all non-proper partitions. 
\end{claim}

%Claim \ref{C:claim} implies that $b_{\lambda}^*L$ is effective boundary , thus finishing the proof of Lemma \ref{L:lemma}.

We consider three cases:

\subsubsection*{Case 1:} If $k-1$ and $k$ belong to the same part, say $J$, then 
\[
\sum_{i\in I, j\in J} w(i,j) = \sum_{i\in I, j\in \{1,\dots,k-1\}\setminus I } \widetilde{w}(i,j) \geq f(\sum_{t\in I} \mu_t)= f(\sum_{t\in I} \lambda_t),
\]
where we used Inequality \eqref{E:tilde-lambda}. In particular, the equality holds when $I$ is a singleton.
%This applies also if $I$ is a singleton.

\subsubsection*{Case 2:} If $J=\{k-1\}$, or $J=\{k\}$,
\[
\sum_{i\in I, j\in J} w(i,j)=w(k-1,k)+\frac{1}{2}\sum_{j\in \{1,\dots,k-2\}} \widetilde{w}(j, k-1)
=f(\lambda_k)-\frac{1}{2}f(2\lambda_{k}) +\frac{1}{2}f(\mu_{k-1})=f(\lambda_k).
\]

\subsubsection*{Case 3:} Suppose $I=I'\cup \{k-1\}$ and $J=J'\cup \{k\}$, where $I'\sqcup J'=\{1,\dots, k-2\}$. Then
\begin{multline}\label{E:m}
\sum_{i\in I, j\in J} w(i,j)=f(\lambda_k)-\frac{1}{2}f(2\lambda_{k})+\frac{1}{2}\sum_{i\in I', j\in J'\cup \{k-1\}} \widetilde{w}(i,j)
+\frac{1}{2}\sum_{i\in I'\cup \{k-1\}, j\in J'} \widetilde{w}(i,j) \\
\geq f(\lambda_k)-\frac{1}{2}f(2\lambda_{k})+\frac{1}{2}f\bigl(\sum_{t\in I'} \lambda_t\bigr)+\frac{1}{2}f\bigl(\sum_{t\in J'} \lambda_t\bigr),
\end{multline}
where we used Inequality \eqref{E:tilde-lambda}.

Denoting $A:=\sum_{t\in I'} \lambda_t$, $B:=\sum_{t\in J'} \lambda_t$, and $a:=\lambda_{k-1}=\lambda_k$, we have $A+B+2a=n$, and \eqref{E:m}
translates into 
\[
\sum_{i\in I, j\in J} w(i,j) \geq f(a)-\frac{1}{2}f(2a)+\frac{1}{2}f(A)+\frac{1}{2}f(B).
\]
Applying the F-inequality \eqref{E:F-nef}: 
\[
\langle f, (A,B,a)\rangle = f(A)+f(B)+2f(a)-2f(A+a)-f(2a)\geq 0,
\] we conclude that
\[
\sum_{i\in I, j\in J} w(i,j) \geq f(A+a)=f\bigl( \sum_{t\in I} \lambda_t\bigr), \quad \text{as desired}.
\]

\end{proof}

\section{Proof of Theorem~\ref{n=8}}\label{S:n8-computation}

To prove the Strong F-conjecture for $\M_{0,8}$, we dualize to the curve space $N_1(\M_{0,8})$, where the conjecture asks whether the cone of curve classes meeting all proper boundary divisors non-negatively is contained in the cone of F-curves. Directly enumerating the extremal rays of the former cone is already infeasible at $n=8$. Instead, we show that if containment fails, some extremal ray must have \emph{non-critical} support (see Definition~\ref{D:critical}). We then rule out all non-critical supports of non-zero classes in this cone by an $S_8$-equivariant branch-and-prune search.

\subsection{Duality and pairwise balanced designs}
Let $\mathcal D_n$ be the set of proper $2$-part partitions $I\sqcup J=[n]$ (so $|I|,|J|\geq 2$ and $|\mathcal D_n|=2^{n-1}-n-1$). Every numerical curve class $C\in N_1(\M_{0,n})$ is determined by its boundary intersection numbers $a_{I,J}=\Delta_{I,J}\cdot C$ on $\mathcal D_n$.

Fix a marking $t\in[n]$, and represent each proper partition by the part $B$ omitting $t$, with weight $a_B=a_{B,[n]\setminus B}$. It follows
from Keel's relations \eqref{keel-1}--\eqref{keel-2} that a weighting $\{a_B\}_{B\in\mathcal D_n}$ arises from a numerical curve class $C\in N_1(\M_{0,n})$ 
if and only if the pair sum $\sum_{B\supseteq\{i,j\}} a_B$ is independent of the pair $\{i,j\}\subset[n]\setminus\{t\}$, with common value $r_t=\psi_t\cdot C$.

\begin{defn}
Fix a marking $t\in[n]$. Non-negative real weights $\{a_B\}$ on the blocks $B\subset[n]\setminus\{t\}$ ($2\leq |B|\leq n-2$) form a \emph{weighted pairwise balanced design} (PBD) of index $r_t$ if
\[
    \sum_{B\supseteq\{i,j\}} a_B = r_t
    \qquad\text{for every distinct pair } i,j\neq t.
\]
The \emph{PBD cone} $Q_n=\PBDcone(n)\subset N_1(\M_{0,n})$ consists of the curve classes meeting every proper boundary divisor non-negatively; equivalently, those whose boundary intersections form a weighted PBD for some (hence every) choice of $t$. We call elements of $Q_n$ \emph{effective designs}.
\end{defn}

One may thus view an effective design as a non-negative weighting of the proper $2$-part partitions of $[n]$, where choosing a root marking $t$ produces a block design 
on $[n]\setminus\{t\}$ (different choices of $t$ can yield non-isomorphic block designs for the same weighting). 
By definition, $Q_n$ is dual to the effective boundary cone in $\Pic(\M_{0,n})\otimes\RR$, 
and carries a natural $S_n$-action permuting $\mathcal D_n$. Let $P_n\subset N_1(\M_{0,n})$ 
be the cone generated by the F-curve classes, which is dual to the F-nef cone. In the range $n\leq 11$, Conjecture~\ref{Cstrong} takes the following dual form.

\begin{conjecture}[Dual Strong F-conjecture]\label{DualCstrong}
For $n\leq 11$, we have $Q_n\subseteq P_n$. Equivalently, every effective design is a non-negative linear combination of F-curves.
\end{conjecture}

For $n=12$, distinguishing the twelfth marking in Pixton's counterexample~\cite{pixton} yields an $(11,5,2)$-biplane, an index-$2$ PBD with $11$ blocks of size~$5$, 
lying in $Q_{12}\setminus P_{12}$.

\subsection{Reduction to non-critical supports}
In principle, one could verify $Q_n\subseteq P_n$ by listing all extremal rays of $Q_n$ and checking that each lies in $P_n$. However, the number of extremal rays grows rapidly: up to $S_n$-symmetry, $Q_6$ has $22$ extremal rays, whereas $Q_7$ already has $59{,}448$ (see \cite{linprog} for the general complexity of polyhedral cone containment\footnote{We are grateful to Yurii Malitskyi for the reference.}). To avoid ray enumeration, we constrain the support of a worst counterexample.

\begin{defn}\label{D:critical}
The \emph{support} of an effective design $a\in Q_n$ is
\[
    \supp(a)=\{B\in\mathcal D_n \mid \Delta_B\cdot a > 0\}.
\]
A subset (or \emph{cell}) $J\subseteq\mathcal D_n$ \emph{supports} a design if some non-zero $a\in Q_n$ has $\supp(a)\subseteq J$. We say that $J$ is \emph{critical} if there exists a non-zero $v\in P_n$ such that $\Delta_B\cdot v \leq 0$ for every $B\notin J$, and \emph{non-critical} otherwise.
\end{defn}

\begin{remark}
By Farkas' lemma, a cell $J$ is non-critical if and only if there exists an effective boundary divisor supported outside $J$ that meets every F-curve strictly positively. In particular, non-criticality is hereditary (every subset of a non-critical cell is non-critical), and criticality passes to supersets.

Geometrically, if $J=\supp(a)$ is critical with witness $v\in P_n\setminus\{0\}$, then $a-\varepsilon v$ remains in $Q_n$ for all sufficiently small $\varepsilon>0$: outside $J$, subtracting a non-positive intersection keeps coordinates non-negative, while inside $J$, coordinates start strictly positive and stay positive.
\end{remark}

\begin{prop}\label{P:n8-ray}
If $Q_n\not\subseteq P_n$, then $Q_n$ has an extremal ray outside $P_n$ whose support is non-critical.
\end{prop}

\begin{proof}
Consider the well-known ample divisor class $q=\psi-\delta$. By Keel's relations \eqref{keel-1}--\eqref{keel-2}, $q$ is a strictly positive linear combination of proper boundary divisors, so $q\cdot a>0$ for every non-zero $a\in Q_n$. Moreover, $q\cdot F=1$ for every F-curve $F$, whence $q\cdot v>0$ for every non-zero $v\in P_n$.

Since $Q_n\not\subseteq P_n$, there is an F-nef divisor class $D$ (non-negative on $P_n$) that is negative somewhere on $Q_n$. The slice $\{a\in Q_n \mid q\cdot a = 1\}$ is a compact polytope; choose a vertex $a$ minimizing $D$ on it. Then $D\cdot a<0$, and $a$ spans an extremal ray of $Q_n$ lying outside $P_n$.

If $\supp(a)$ were critical with witness $v\in P_n\setminus\{0\}$, then for sufficiently small $\varepsilon>0$ (with $\varepsilon(q\cdot v)<1$), the renormalized class $a'=\frac{a-\varepsilon v}{1-\varepsilon(q\cdot v)}$ would still lie in the slice and satisfy
\[
    D\cdot a'
    = \frac{D\cdot a - \varepsilon(D\cdot v)}{1-\varepsilon(q\cdot v)}
    \leq \frac{D\cdot a}{1-\varepsilon(q\cdot v)}
    < D\cdot a,
\]
contradicting the minimality of $a$. Thus $\supp(a)$ is non-critical.
\end{proof}

Proposition~\ref{P:n8-ray} and duality reduce Theorem~\ref{n=8} to the following finite combinatorial statement.

\begin{prop}\label{P:n8-cells}
For $n\leq 8$, every cell $J\subseteq\mathcal D_n$ is either critical or supports no non-zero effective design.
\end{prop}

For $n\leq 7$, Proposition~\ref{P:n8-cells} follows from the Strong F-conjecture~\cite{FG,larsen}: any non-zero effective design already lies in $P_n$ and hence witnesses the criticality of any cell containing its support. One can also verify $n=7$ directly by the same non-critical-cell exhaustion used below, yielding an independent proof of the Strong F-conjecture for $\M_{0,7}$.
%\footnote{For $n=7$, the same deadend exhaustion leaves $7{,}965$ non-empty frontier orbits in total, with $3$ surviving at size~$8$; the single covered size-$9$ candidate is critical, so $\mathcal L_9=\varnothing$.}
We now prove Proposition~\ref{P:n8-cells} for $n=8$.

\subsection{Branch-and-prune enumeration for eight markings}\label{S:n8-enum}
For $n=8$, there are $|\mathcal D_8|=119$ proper $2$-part partitions and $1{,}701$ F-curves. Our goal is to rule out any non-critical cell $S\subseteq\mathcal D_8$ supporting a non-zero effective design $a\in Q_8$.

We build a sequence of frontiers $\mathcal L_0,\mathcal L_1,\dots$, where each $\mathcal L_k$ is a list of $S_8$-orbits of non-critical $k$-element cells in $\mathcal D_8$. Starting from $\mathcal L_0=\{[\varnothing]\}$, we pass from $\mathcal L_k$ to $\mathcal L_{k+1}$ in four steps:

\begin{enumerate}[label=\arabic*.,leftmargin=*,itemsep=.25em]
\item \textbf{Extension and downward closure.} For each orbit representative $C$ in $\mathcal L_k$ and each $d\in\mathcal D_8\setminus C$, form the $(k+1)$-element cell $S=C\cup\{d\}$ and record its $S_8$-orbit $[S]$. Since every subset of a non-critical support is non-critical, keep a candidate orbit $[S]$ only if all $k+1$ of its $k$-element deletions $[S\setminus\{j\}]$ belong to $\mathcal L_k$.

\item \textbf{Criticality pruning.} Discard every candidate orbit $[S]$ that is critical. By the dual characterization above, testing criticality is a single linear feasibility check: either an effective boundary divisor supported outside $S$ meets every F-curve strictly positively (so $S$ survives), or a non-zero non-negative combination of F-curves witnesses criticality (so $[S]$ is discarded).

\item \textbf{Deadend pruning.} Let $C$ be a non-empty parent in $\mathcal L_k$, and let $T(C)$ be the set of partitions $d\in\mathcal D_8\setminus C$ for which the child orbit $[C\cup\{d\}]$ is still alive. If $C$ lay inside the support $S$ of a non-critical effective design $a$ whose $(k+1)$-subsets were all still alive, then every $d\in S\setminus C$ would belong to $T(C)$, so $C\cup T(C)$ would support an effective design that is \emph{strictly positive on all of $C$}. We test this feasibility condition for every non-empty parent $C\in\mathcal L_k$. If no such design exists, we declare $C$ a \emph{deadend} and discard every surviving child orbit extending $C$. Removing child orbits shrinks $T(C')$ for other parents, which may turn them into deadends as well; we iterate these pruning waves until no further parents become deadends. The surviving $(k+1)$-orbits form $\mathcal L_{k+1}$.

\item \textbf{Design test.} Test every orbit $[J]\in\mathcal L_{k+1}$ to determine whether $J$ supports a non-zero effective design (on all of $J$ or on a proper subset). Stop when $\mathcal L_{k+1}$ is empty.
\end{enumerate}

\smallskip\noindent\textbf{Completeness.}
Suppose for contradiction that a non-zero effective design $a\in Q_8$ has non-critical support $S\subseteq\mathcal D_8$. We claim that $[D]\in\mathcal L_k$ for every $k\leq |S|$ and every $k$-subset $D\subseteq S$. The base case $k=0$ is automatic. Assuming the claim holds at size $k<|S|$, every $(k+1)$-subset of $S$ has all its $k$-subsets in $\mathcal L_k$ and is non-critical (as a subset of $S$), so it survives Steps~1 and~2.

Suppose a wave of Step~3 is the first to remove the orbit of some $(k+1)$-subset of $S$. Relabel the deadend parent $C$ responsible for that removal so that both $C$ and the removed child lie in $S$. At the start of that wave, no $(k+1)$-subset of $S$ has yet been discarded, so every $d\in S\setminus C$ lies in $T(C)$, giving $S\subseteq C\cup T(C)$. But $a$ itself is an effective design supported in $C\cup T(C)$ and strictly positive on all of $C$, contradicting that $C$ is a deadend.

By induction, every subset of $S$ survives; in particular, $[S]$ reaches $\mathcal L_{|S|}$ and is detected in Step~4. Conversely, if Step~4 finds no design on any frontier and $\mathcal L_m=\varnothing$ for some $m$, then no non-critical cell of any size can support a non-zero effective design.

\smallskip\noindent\textbf{Certification and formal verification.}
Every critical rejection in Step~2, every deadend removal in Step~3, and every frontier exclusion in Step~4 is certified by an exact integer witness (an integer combination of F-curves for critical rejections, and an integer separating functional for deadend removals and design-free frontier cells), eliminating any floating-point ambiguity. Moreover, we have formally verified in Lean both Proposition~\ref{P:n8-cells} for $n=8$ and the resulting rational divisor statement of Theorem~\ref{n=8} in Keel's presentation of $\Pic(\M_{0,8})$. All code, computational outputs, exact certificates, and the Lean formalization are publicly available at
\begin{center}
{\small\url{https://github.com/maksym-fedorchuk-math/maksym-fedorchuk-math-N8_full_lp_certificates}}.
\end{center}

\subsection{Computational results and conclusion of the proof}
Running the algorithm on $\mathcal D_8$ produces the counts in the table below. At each size $s$, the columns record the number of size-$s$ candidate orbits passing the downward-closure check (Step~1), how many of those are critical (Step~2), how many size-$(s-1)$ parents are declared deadends during the ensuing pruning waves (Step~3), the size of the resulting frontier $\mathcal L_s$, and the number of frontier orbits supporting a non-zero effective design (Step~4).

\begin{center}\small
\renewcommand{\arraystretch}{1.07}
\setlength{\tabcolsep}{5pt}
\begin{tabular}{@{}rrrrrr@{}}
\hline
$s$ & \shortstack[r]{Covered\\orbits} & Critical & \shortstack[r]{Deadends} & \shortstack[r]{Frontier\\orbits} & \shortstack[r]{With a\\design} \\ \hline
1 & 3 & 0 & 0 & 3 & 0 \\
2 & 14 & 0 & 0 & 14 & 0 \\
3 & 98 & 5 & 0 & 93 & 0 \\
4 & 842 & 7 & 0 & 835 & 0 \\
5 & 9,133 & 12 & 0 & 9,121 & 0 \\
6 & 109,715 & 55 & 0 & 109,660 & 0 \\
7 & 1,297,834 & 225 & 12 & 1,297,444 & 0 \\
8 & 13,982,998 & 733 & 343 & 13,973,757 & 0 \\
9 & 130,677,650 & 2,839 & 13,900,790 & 109,166 & 0 \\
10 & 127,799 & 41 & 107,481 & 790 & 0 \\
11 & 271 & 9 & 547 & 74 & 0 \\
12 & 15 & 1 & 5 & 14 & 0 \\
13 & 2 & 2 & 14 & 0 & 0 \\
\hline
Total & 146,206,374 & 3,929 & 14,009,192 & 15,500,971 & 0 \\ \hline
\end{tabular}
\end{center}

The search peaks at size~$9$ with over $130$~million covered orbits, where deadend pruning becomes decisive: eliminating $13{,}900{,}790$ of the $13{,}973{,}757$ size-$8$ parents collapses the frontier to $109{,}166$ orbits, and the remaining levels shrink rapidly. Across all sizes, $3{,}929$ covered orbits are critical (each inclusion-minimal among critical cells, since all of its proper subsets were already non-critical). At size~$13$, both covered candidates are critical, leaving all $14$ size-$12$ parents as deadends and giving $\mathcal L_{13}=\varnothing$. None of the $15{,}500{,}971$ non-empty frontier orbits across sizes $1\leq s\leq 12$ supports a non-zero effective design.

\begin{proof}[Conclusion of the proof of Theorem~\ref{n=8}]
If a non-critical cell supported a non-zero effective design with support $S\subseteq\mathcal D_8$, completeness of the search would force an orbit in $\mathcal L_{13}$ when $|S|\geq 13$, or a positive design test on $[S]\in\mathcal L_{|S|}$ when $|S|\leq 12$. Since $\mathcal L_{13}=\varnothing$ and no orbit in $\bigcup_{s=1}^{12}\mathcal L_s$ supports a design, Proposition~\ref{P:n8-cells} holds for $n=8$. Proposition~\ref{P:n8-ray} then yields $Q_8\subseteq P_8$. By duality, every F-nef divisor class lies in the real cone generated by the proper boundary divisors; since the boundary cone and Keel's relations are defined over $\QQ$, every rational F-nef divisor is $\QQ$-linearly equivalent to an effective combination of boundary divisors. %The F-conjecture for $\M_{0,8}$ then follows by restriction to the boundary, using the known cases $n\leq 7$.
\end{proof}

\label{S:n8-end}

\newcommand{\netarxiv}[1]{\href{http://arxiv.org/abs/#1}{{\sffamily{\texttt{arXiv:#1}}}}}

%\linespread{1}\normalfont\selectfont

%\renewcommand{\bibliofont}{\normalfont\small}
\bibliographystyle{alpha} 
% Typeset bibliography supplied with the matching source archive.

\end{document}